\documentclass[10pt]{article}
\font\smallit=cmti10
\font\smalltt=cmtt10

\usepackage{caption}
\usepackage{color}
\usepackage{amssymb,latexsym,amsmath,epsfig,amsthm} 
\usepackage{empheq}
\usepackage{multicol}
\usepackage{here}
\usepackage{url}
\usepackage{ascmac}
\makeatletter
\usepackage{here}
\usepackage{comment}

\renewcommand\section{\@startsection {section}{1}{\z@}
{-30pt \@plus -1ex \@minus -.2ex}
{2.3ex \@plus.2ex}
{\normalfont\normalsize\bfseries\boldmath}}

\renewcommand\subsection{\@startsection{subsection}{2}{\z@}
{-3.25ex\@plus -1ex \@minus -.2ex}
{1.5ex \@plus .2ex}
{\normalfont\normalsize\bfseries\boldmath}}

\renewcommand{\@seccntformat}[1]{\csname the#1\endcsname. }

\makeatother
\newtheorem{theorem}{Theorem}
\newtheorem{lemma}{Lemma}

\newtheorem{corollary}{Corollary}
\theoremstyle{definition}
\newtheorem{definition}{Definition}
\newtheorem{remark}{Remark}
\newtheorem{example}{Example}

\begin{document}

\begin{center}
\uppercase{\bf A Variant of Wythoff's Game Defined by Hofstadter's G-Sequence}
\vskip 20pt
{\bf  Kahori Komaki  }\\
{\smallit Keimei Gakuin Junior and High School, Kobe City, Japan}\\
{\tt xiaomukahori@gmail.com}
\vskip 10pt
{\bf Ryohei Miyadera}\\
{\smallit Keimei Gakuin Junior and High School, Kobe City, Japan}\\
{\tt runnerskg@gmail.com}
\vskip 10pt
{\bf Aoi Murakami }\\
{\smallit Kwansei Gakuin University }\\
{\tt atatpj728786.55986@gmail.com}
\end{center}
\vskip 20pt
\centerline{\smallit Received: , Revised: , Accepted: , Published: } 
\vskip 30pt


\centerline{\bf Abstract}
\noindent
In this paper, we study a variant of the classical Wythoff's game. The classical form is played with two piles of stones, from which two players take turns to remove stones from one or both piles. When removing stones from both piles, an equal number must be removed from each. The player who removes the last stone or stones is the winner.
Equivalently, we consider a single chess queen placed somewhere on a large grid of squares. Each player can move the queen toward the upper-left corner of the grid, either vertically, horizontally, or diagonally in any number of steps. The winner is the player who moves the queen to the upper-left corner, position $(0,0)$ in our coordinate system. We call $(0,0)$ the terminal position of Wythoff's game. The set of $\mathcal{P}$-positions (previous player's winning positions) is $\{(\lfloor n \phi  \rfloor, \lfloor n \phi  \rfloor + n):n \in \mathbb{Z}_{\geq0} \}$ $\cup \{( \lfloor n \phi  \rfloor +n ,\lfloor n \phi  \rfloor):n \in \mathbb{Z}_{\geq0} \}$, where $\phi= \frac{1+\sqrt{5}}{2}$. In our variant of Wythoff's game, we have a set of positions $\{(x,y):x+y \leq 2\}$ $=\{(0,0),(1,0),(0,1),(1,1),(2,0),(0,2)\}$ as the terminal set. The player who moves in this terminal set is the winner.
For this variant, we define a function $g$ by 
$g(n)=1-g(h(n-1))$ if $h(n-2) < h(n-1)$ and $g(n)=1$ if not, where $h$ is Hofstadter's G-sequence.
Then, the set of $\mathcal{P}$-positions is $\{(\lfloor n \phi \rfloor+g( n )-1, \lfloor n (\phi + 1) \rfloor  +g(n)):n \in \mathbb{Z}_{\geq0} \}$ $\cup \{( \lfloor n (\phi + 1) \rfloor  +g(n),\lfloor n \phi \rfloor+g(n)-1):n \in \mathbb{Z}_{\geq0} \}$.
This variant has two remarkable properties.
For a position $(x,y)$ with $x \geq 8$ or $y \geq 8$, the Grundy number of the position $(x,y)$ is $1$ if and only if $(x,y)$ is a $\mathcal{P}$-position in Wythoff's game.
Another remarkable property is that for a position $(x,y)$ with $x \geq 8$ or $y \geq 8$,  $(x,y)$ is a $\mathcal{P}$-position of the misère version of this variant if and only if   $(x,y)$ is a $\mathcal{P}$-position of Wythoff's game.

 \pagestyle{myheadings} 
 \markright{\smalltt \hfill} 
 \thispagestyle{empty} 
 \baselineskip=12.875pt 
 \vskip 30pt

\section{Introduction to Wythoff's Game and its Variant} 
First, we define Wythoff's game. For the details of Wythoff's game, see \cite{wythoffpaper}.
Let $\mathbb{Z}_{\geq0}$ and $\mathbb{N}$ be the sets of non-negative numbers and natural numbers, respectively. 
\begin{definition}\label{wythoff}
Wythoff's game is played with two piles of stones. Two players take turns removing stones from one or both piles. When removing stones from both piles, the number of stones removed from each pile should be equal. The player who removes the last stone or stones wins.
An equivalent description of the game is that a single chess queen is placed somewhere on a large grid of squares, and each player can move the queen towards the upper-left corner of the grid, either vertically, horizontally, or diagonally, for any number of steps.  The winner is the player who moves the queen to the upper-left corner.
\end{definition}

Many people have proposed many variants of Wythoff's game, and the author of the present article also presented a variant in \cite{suetsugu2020} and \cite{integer2025}.
We define a new variant of Wythoff's game.
\begin{definition}\label{wythoffvar}
This variant of Wythoff's game is played like the classical Wythoff’s game, with the terminal set 
 $\{(x,y):x+y \leq 2\}$ $=\{(0,0),(1,0),(0,1),(1,1),(2,0),$ \\
 $(0,2)\}$. 
\end{definition}

Figure~\ref{moveofqueen}  shows the moves that the queen can make, and 
Figures ~\ref{chessboard} and \ref{chessboard2} show the terminal positions of the classical Wythoff's game and the game in Definition \ref{wythoffvar}, respectively.

\begin{figure}[H]
\begin{tabular}{ccc}
\begin{minipage}[t]{0.33\textwidth}
\begin{center}
\includegraphics[height=2.cm]{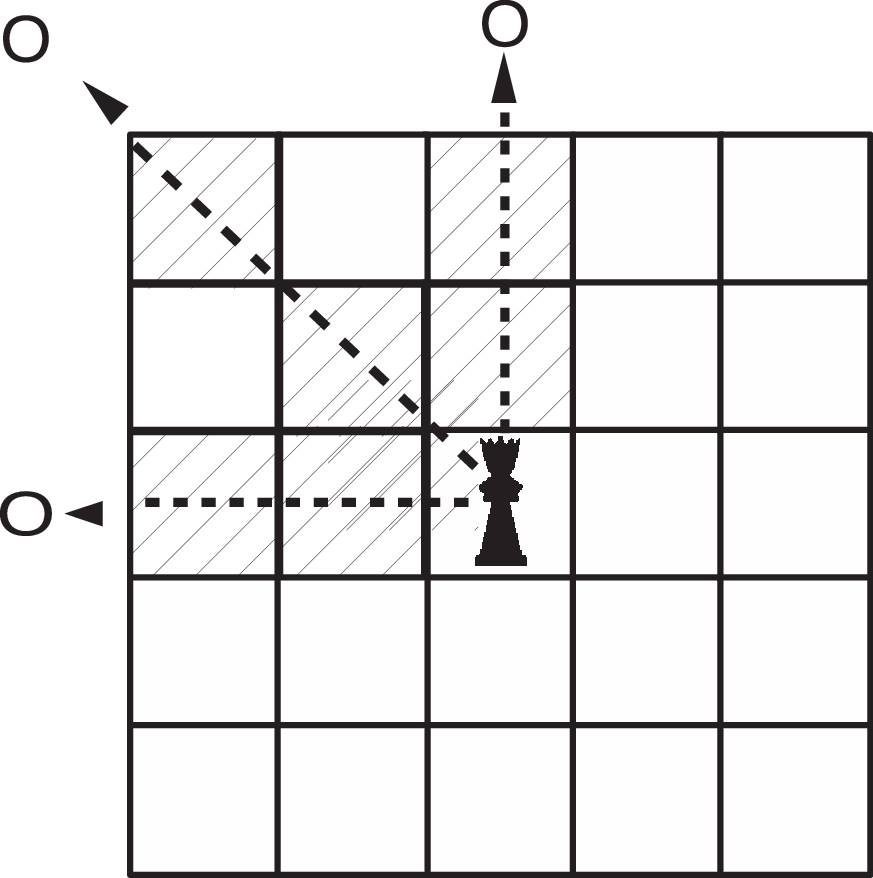}
\captionsetup{labelsep = period}
\caption{The moves of \\ the queen}
\label{moveofqueen}
\end{center}
\end{minipage}
\begin{minipage}[t]{0.33\textwidth}
\begin{center}
	\includegraphics[height=2.5cm]{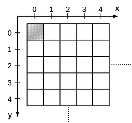}
\captionsetup{labelsep = period}
\caption{The terminal \\ position of Wythoff's \\game}
\label{chessboard}
\end{center}
\end{minipage}
\begin{minipage}[t]{0.33\textwidth}
\begin{center}
\includegraphics[height=2.5cm]{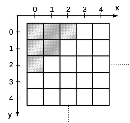}
\captionsetup{labelsep = period}
\caption{The terminal \\ positions of the game \\ in Definition \ref{wythoffvar}}
\label{chessboard2}
\end{center}
\end{minipage}
\end{tabular}
\end{figure}

We define $\textit{move}(x,y)$ in Wythoff's game and its variant in Definition \ref{movewythoff}.
\begin{definition}\label{movewythoff}
$\textit{move}(x,y)$ is the set of all positions that can be reached from $(x,y)$.
 For any $x,y \in Z_{\ge 0}$, let
	\begin{flalign}
	 & \textit{move}(x,y)	=M_1(x,y) \cup M_2(x,y) \cup M_3(x,y),\text{where } & & \nonumber \\
		& M_1(x,y)= \{(u,y):u<x \text{ and } u \in Z_{\ge 0}\}, M_2(x,y)=\{(x,v):v<y \text{ and } v \in Z_{\ge 0}\},& \nonumber \\
		& \text{ and } & \nonumber \\
		& M_3(x,y)=\{(x-t,y-t): 1 \leq t \leq \min(x,y) \text{ and } t \in Z_{\ge 0}\}.& \nonumber
	\end{flalign}
\end{definition}

\begin{remark}\label{explainm1m2m3}
$M_1(x,y), M_2(x,y), M_3(x,y)$ are the sets of horizontal, vertical, and diagonal moves, respectively.
$M_3(x,y)$ is an empty set if $x = 0$ or $y = 0$.
For Wythoff's game $\textit{move}(0,0)= \emptyset$, and for the game in Definition \ref{wythoffvar}, 
$\textit{move}(x,y)= \emptyset$ for $x, y \in \mathbb{Z}_{\geq0}$ such that $x+y \leq 2$.
\end{remark}

\section{Combinatorial Game Theory Definitions and a Theorem}
For completeness, we briefly review some of the necessary concepts in combinatorial game theory by referring to $\cite{lesson}$ and \cite{combysiegel}.

Wythoff's game is an impartial game without drawings; only two outcome classes are possible.
\begin{definition}\label{NPpositions}
A position is referred to as a P-position if it is the winning position for the previous player (the player who has just moved), as long as the player plays correctly at each stage. A position is referred to as an N-position if it is the winning position for the next player, as long as they play correctly at each stage.
\end{definition}

\begin{definition}\label{sumofgames}
	The \textit{disjunctive sum} of the two games, denoted by $\mathbf{G}+\mathbf{H}$, is a super game in which a player may move either in $\mathbf{G}$ or $\mathbf{H}$ but not in both.
\end{definition}

\begin{definition}\label{defofmex}
The \textit{minimum excluded value} ($\textit{mex}$) of a set $S$ of nonnegative integers is the least nonnegative integer that is not in S.
\end{definition}

\begin{definition}\label{defofmexgrundy}
Let $\mathbf{p}$ be a position in the impartial game. The associated \textit{Grundy number} is denoted by $G(\mathbf{p})$ and is 
 recursively defined by 
	$G(\mathbf{p}) = \textit{mex}(\{G(\mathbf{h}): \mathbf{h} \in \textit{move}(\mathbf{p})\}).$
\end{definition}

The next result demonstrates the usefulness of the Sprague--Grundy theory for impartial games.
\begin{theorem}[\cite{lesson}]\label{theoremofsumg}
Let $\mathbf{G}$ and $\mathbf{H}$ be impartial rulesets, and $G_{\mathbf{G}}$ and $G_{\mathbf{H}}$ be the Grundy numbers of game $\mathbf{g}$ played under the rules of $\mathbf{G}$ and game $\mathbf{h}$ played under those of $\mathbf{H}$. Then, we obtain the following:\\
	$(i)$ for any position $\mathbf{g}$ in $\mathbf{G}$, we have that 
	$G_{\mathbf{G}}(\mathbf{g})=0$ if and only if $\mathbf{g}$ is the $\mathcal{P}$-position; \\
	$(ii)$ the Grundy number of positions $\{\mathbf{g},\mathbf{h}\}$ in game $\mathbf{G}+\mathbf{H}$ is
	$G_{\mathbf{G}}(\mathbf{g})\oplus G_{\mathbf{H}}(\mathbf{h})$.
\end{theorem}
Using Theorem \ref{theoremofsumg}, we can determine the $\mathcal{P}$-position by calculating the Grundy numbers and the $\mathcal{P}$-position of the sum of the two games by calculating the Grundy numbers of the two games.


\section{Wythoff's Game and the Lower and Upper Wythoff Sequences}
In this section, we review a theorem of Wyhoff's game, and present some lemmas
for later use. In the remainder of this paper, we let $\phi= \frac{1+\sqrt{5}}{2}$.

\begin{definition}\label{wythoffpp}
Let 
\begin{equation}
P_{0,1}=\{(\lfloor n \phi \rfloor, \lfloor n (\phi + 1) \rfloor ):n \in \mathbb{Z}_{\geq0} \}, \nonumber    
\end{equation}
\begin{equation}
P_{0,2}= \{( \lfloor n (\phi + 1) \rfloor  ,\lfloor n \phi \rfloor):n \in \mathbb{Z}_{\geq0} \}, \nonumber    
\end{equation}
and
\begin{equation}
P_0=P_{0,1} \cup P_{0,2}. \nonumber    
\end{equation}
\end{definition}

\begin{theorem}[\cite{wythoffpaper}]\label{theoremforwythoff}
 The set $P_0$ in Definition \ref{newfunction} is the set of $\mathcal{P}$-positions of Wythoff's game.
\end{theorem}

\begin{definition}\label{definitionofa1a2}
Let $A_1=\{a_1(n)=\lfloor n \phi \rfloor :n \in \mathbb{Z}_{\geq0}\}$ and $A_2=\{a_2(n) =\lfloor n (\phi + 1) \rfloor  : n \in \mathbb{Z}_{\geq0}\}$.
We call $A_1$ and $A_2$ the lower and upper Wythoff sequences.
\end{definition}

\begin{lemma}\label{rayleigh}
The lower and upper Wythoff sequences $A_1$ and $A_2$ satisfy the following:\\
$(i)$  Sets $A_1$ and $A_2$ satisfy  $A_1 \cup A_2 = \mathbb{Z}_{\geq0}$;\\
$(ii)$ Sets $A_1$ and $A_2$ satisfy $A_1 \cap A_2 = \{0\}$.
\end{lemma}
\begin{proof}
This follows directly from the well-known Rayleigh theorem.
\end{proof}

\begin{lemma}\label{lemmaphi12}
For the lower Wythoff sequence and the upper sequence, we have the following:\\
$(i)$ for $n \in \mathbb{Z}_{\geq0}$, $a_1(n+1)-a_1(n)$ $= \lfloor (n+1)\phi  \rfloor - \lfloor n \phi \rfloor$  = $1$ or $2$;\\
$(ii)$ for $n \in \mathbb{Z}_{\geq0}$, $a_1(n+2)-a_1(n)$ $= \lfloor (n+2) \phi  \rfloor - \lfloor n \phi \rfloor$ = $3$ or $4$;\\
$(iii)$ for $m \in \mathbb{Z}_{\geq0}$, $a_2(m+1)-a_2(m)$ $= \lfloor (m+1)(\phi +1)  \rfloor - \lfloor m (\phi +1) \rfloor$  = $2$ or $3$. 
\end{lemma}
\begin{proof}
For any positive real numbers $x,y$, 
$\lfloor x+y \rfloor - \lfloor y \rfloor$ = $\lfloor x \rfloor$ or $\lfloor x \rfloor + 1$. Hence, 
$\lfloor (n+1) \phi  \rfloor - \lfloor n \phi \rfloor$ = $\lfloor  \phi \rfloor =1$ or $\lfloor (n+1) \phi  \rfloor - \lfloor n \phi \rfloor$ =$\lfloor  \phi \rfloor + 1 =2$, and $\lfloor (n+2) \phi  \rfloor - \lfloor n \phi \rfloor$ = $\lfloor 2 \phi \rfloor =3$ or $\lfloor (n+2) \phi  \rfloor - \lfloor n \phi \rfloor$ = $\lfloor 2 \phi \rfloor + 1 =4$.
Similarly, 
$a_2(m+1)-a_2(m)$ $= \lfloor (m+1)(\phi +1)  \rfloor - \lfloor m (\phi +1) \rfloor$   $= \lfloor (m+1)\phi  \rfloor - \lfloor m \phi \rfloor$ + 1= 2 or 3.
\end{proof}

\begin{lemma}\label{nn1eqx}
For any  $x \in \mathbb{Z}_{\geq0}$, there exists $n \in \mathbb{Z}_{\geq0}$ such that 
$x = \lfloor n \phi \rfloor$ or $x = \lfloor n \phi \rfloor -1$.
\end{lemma}
\begin{proof}
For any  $x \in \mathbb{Z}_{\geq0}$, there exists $n \in \mathbb{Z}_{\geq0}$ such that 
$\lfloor (n-1) \phi  \rfloor < x \leq \lfloor n \phi \rfloor$.
By $\mathrm{(i)}$ of Lemma \ref{lemmaphi12}, $x = \lfloor n \phi \rfloor$ or $x = \lfloor n \phi \rfloor -1$.
\end{proof}

\begin{lemma}\label{twocases}
For any $n \in \mathbb{N}$ with $n \geq 2$, one of the following statements is true:\\
$(i)$ there exists $m \in \mathbb{Z}_{\geq0}$ such that 
$\lfloor m (\phi + 1) \rfloor $, $\lfloor (n-1) \phi \rfloor$, $\lfloor n \phi \rfloor$ are three consecutive numbers;\\
$(ii)$ there exists $m \in \mathbb{Z}_{\geq0}$ such that 
$\lfloor (n-1) \phi \rfloor$, $\lfloor m (\phi + 1) \rfloor $, $\lfloor n \phi \rfloor$ are three consecutive numbers.
\end{lemma}
\begin{proof}
By $\mathrm{(i)}$ of Lemma \ref{lemmaphi12}, 
$\lfloor n \phi \rfloor - \lfloor (n-1) \phi \rfloor$ = 1 or 2.\\
$\mathrm{(i)}$ If $\lfloor n \phi \rfloor - \lfloor (n-1) \phi \rfloor =1$, then
by $\mathrm{(i)}$ and $\mathrm{(ii)}$ of Lemma \ref{lemmaphi12}, $\lfloor (n-2) \phi \rfloor = \lfloor (n-1) \phi \rfloor-2$. Hence by Lemma \ref{rayleigh} there exists $m \in \mathbb{Z}_{\geq0}$ such that 
$\lfloor m (\phi + 1) \rfloor $, $\lfloor (n-1) \phi \rfloor$, $\lfloor n \phi \rfloor$ are three consecutive numbers.\\
$\mathrm{(ii)}$ If $\lfloor n \phi \rfloor-\lfloor (n-1) \phi \rfloor =2$, then by Lemma \ref{rayleigh}, 
there exists $m \in \mathbb{Z}_{\geq0}$ such that $\lfloor (n-1) \phi \rfloor$, $\lfloor m (\phi + 1) \rfloor $, $\lfloor n \phi \rfloor$ are three consecutive numbers.
\end{proof}

\section{$\mathcal{P}$-Positions of Varinat of Wythoff's Game}
In this section, we study the set of $\mathcal{P}$-Positions of the variant in Definition \ref{wythoffvar}, and 
we need to define a function to describe the set of $\mathcal{P}$-Positions.
\begin{definition}\label{newfunction}
Let $\phi= \frac{1+\sqrt{5}}{2}$.\\
$(i)$ We define a function $g$ as
$g(0)=1$, $g(1)=0$, and 
\begin{equation}
g(n)=
\begin{cases}
1-g(m) & (\mbox{ if } \lfloor n \phi  \rfloor=\lfloor m (\phi + 1)  \rfloor +1 \mbox{ for } m \in \mathbb{Z}_{\geq0}),\\ \nonumber
1 & (\mbox{ else }).\nonumber
\end{cases}
\end{equation} 
$(ii)$ Let 
\begin{equation}
P_{1,1}=\{(\lfloor n \phi \rfloor+g( n )-1, \lfloor n (\phi + 1) \rfloor  +g(n)):n \in \mathbb{Z}_{\geq0} \},  \nonumber  
\end{equation}
\begin{equation}
P_{1,2}= \{( \lfloor n (\phi + 1) \rfloor  +g(n),\lfloor n \phi \rfloor+g(n)-1):n \in \mathbb{Z}_{\geq0} \},\nonumber      
\end{equation}
and 
\begin{equation}
P_1= P_{1,1} \cup P_{1,2}\cup \{(0,0),(1,1)\}.  \nonumber    
\end{equation}
\end{definition}

Figures \ref{pcqueen} and \ref{pvqueen} describe the set $P_0$ and the set $P_1$ respectively, and 
in Figure \ref{pcvqueen}, both sets are displayed.
As shown in Figure \ref{pcvqueen}, there are some relationships between the two sets. The function $g$ in Definition \ref{newfunction} describes the relationships between these sets, and as we discuss later, we can define this function $g$ by the well-known Hofstadter G-sequence.

\begin{figure}[H]
\begin{tabular}{cc}
\begin{minipage}[t]{0.5\textwidth}
\begin{center}
	\includegraphics[height=3.5cm]{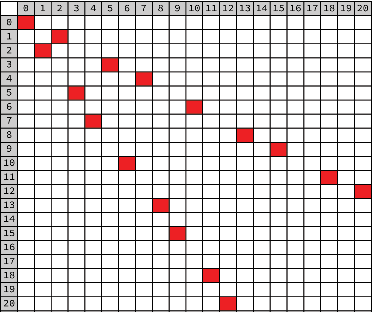}
\captionsetup{labelsep = period}
\caption{$P_0$}
\label{pcqueen}
\end{center}
\end{minipage}
\begin{minipage}[t]{0.5\textwidth}
\begin{center}
\includegraphics[height=3.5cm]{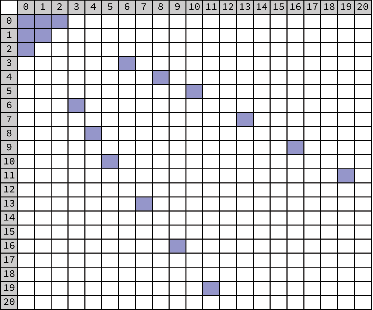}
\captionsetup{labelsep = period}
\caption{$P_1$}
\label{pvqueen}
\end{center}
\end{minipage}
\end{tabular}
\end{figure}

\begin{figure}[H]
\begin{tabular}{c}
\begin{minipage}[t]{0.5\textwidth}
\begin{center}
\includegraphics[height=3.5cm]{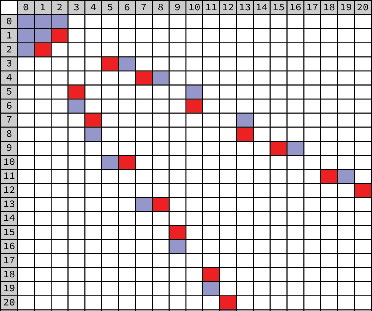}
\captionsetup{labelsep = period}
\caption{$P_0$ and $P_1$}
\label{pcvqueen}
\end{center}
\end{minipage}
\end{tabular}
\end{figure}

\begin{definition}\label{definitionofb1b2}
For $n \in \mathbb{Z}_{\geq0}$, let $B_1=\{b_1(n)=\lfloor n \phi \rfloor+g(n)-1:n \in \mathbb{Z}_{\geq0} \}$ and 
$B_2=\{b_2(n)=\lfloor n (\phi + 1) \rfloor +g(n):n \in \mathbb{Z}_{\geq0} \}$.
\end{definition}

\begin{example}\label{exampleforb1b2}
$(i)$ Here, we caluculate $g(n)$ for $n=0,1,2,3,4$.
By Definition \ref{newfunction}, $g(0)=1$ and $g(1)=0$.
Since $\lfloor 2 \phi \rfloor = 3 = \lfloor (\phi+1) \rfloor +1$,
by Definition \ref{newfunction}, $g(2)=1-g(1)=1$.
Since $\lfloor 3 \phi \rfloor =4$, there is no $m \in \mathbb{Z}_{\geq0}$ such that 
$\lfloor m (\phi+1) \rfloor +1 =4$. Hence
by Definition \ref{newfunction}, $g(3)=1$.
Since $\lfloor 4 \phi \rfloor = 6 = \lfloor 2(\phi+1) \rfloor +1$,
by Definition \ref{newfunction}, $g(4)=1-g(3)=0$.\\
$(i)$ Here, we caluculate $b_1(n)$ and $b_2(n)$ for $n=0,1,2,3,4$.
$b_1(0)=\lfloor 0 \phi \rfloor +g(0)-1=0$, $b_1(1)=\lfloor  \phi \rfloor +g(1)-1=0$,
$b_1(2)=\lfloor 2 \phi \rfloor +g(2)-1=3$, $b_1(3)=\lfloor 3 \phi \rfloor +g(3)-1=4$,
$b_1(4)=\lfloor 4 \phi \rfloor +g(4)-1=5$.

$b_2(0)=\lfloor 0 (\phi +1)\rfloor +g(0)=1$, $b_2(1)=\lfloor  (\phi +1)\rfloor +g(1)=2$,
$b_2(2)=\lfloor  2(\phi +1)\rfloor +g(2)=6$, $b_2(3)=\lfloor  3(\phi +1)\rfloor +g(3)=8$,
$b_2(4)=\lfloor  4(\phi +1)\rfloor +g(4)=10$.
\end{example}

\begin{lemma}\label{compareab}
We have $b_1(n) \leq a_1(n)$ and $b_2(n) \geq a_2(n)$ for $n \in \mathbb{Z}_{\geq0}$.
\end{lemma}
\begin{proof}
This follows directly from Definitions \ref{definitionofa1a2} and \ref{definitionofb1b2}.
\end{proof}

In the remainder of this section, we aim to prove that $P_1$ is the set of 
$\mathcal{P}$-positions of the variant in Definition \ref{wythoffvar}.
We need some lemmas.

 \begin{lemma}\label{fromntoplemma1}
Let $n,m \in \mathbb{N}$ such that 
\begin{equation}
a_1(n) = a_2(m)+1 = a_1(n-1)+2. \label{condition1a}
\end{equation}
Then, one of the following statements is true:\\
$(i)$  $\lfloor n \phi \rfloor  = b_1(n) \in B_1$, 
$\lfloor n \phi \rfloor -1   = b_2(m) \in B_2$, and 
$b_1(n-1) < b_2(m) < b_1(n) < b_2(m+1)$;\\
$(ii)$  
$\lfloor n \phi \rfloor = b_2(m) \in B_2$, $\lfloor n \phi \rfloor -1 = b_1(n) \in B_1$,
$b_1(n-1) < b_1(n) < b_2(m)$, and 
$b_2(m-1) < b_1(n) < b_2(m)$.
\end{lemma}
\begin{proof}
By $\mathrm{(iii)}$ of Lemma \ref{lemmaphi12},
\begin{equation}
a_2(m-1) \leq a_2(m)-2 \label{aam1a2m2}   
\end{equation}
and 
\begin{equation}
a_2(m+1) \geq a_2(m)+2. \label{a2m1a2mp2}   
\end{equation}
By Definition \ref{newfunction}  and Equation $(\ref{condition1a})$, 
\begin{equation}
g(n) = 1 - g(m).  \label{gngmeq2}  
\end{equation}
Then, we have two cases.\\
\noindent {\tt Case 1}: If $g(n)=1$, then by Equation $(\ref{gngmeq2})$, $g(m)=0$. Hence, by Equation $(\ref{condition1a})$, we obtain
\begin{equation}
b_1(n) = a_1(n)+1-g(n)=a_1(n)  = \lfloor n \phi \rfloor \label{b1ncal}
\end{equation}
and
\begin{equation}
 b_2(m)   = a_2(m)+g(m)=a_2(m)=a_1(n)-1 = \lfloor n \phi \rfloor -1. \label{b2mcal}
\end{equation}
By Lemma \ref{compareab} and Equations $(\ref{condition1a})$ and $(\ref{b2mcal})$,
\begin{equation}
b_1(n-1) \leq a_1(n-1) < a_2(m) =b_2(m) \label{ineqfirst}  
\end{equation}
, and by  Equations $(\ref{condition1a})$, $(\ref{b1ncal})$, and $(\ref{b2mcal})$,
\begin{equation}
b_2(m)=a_2(m) < a_1(n) = b_1(n). \label{ineqsecond}   
\end{equation}
By Lemma \ref{compareab} and Equations $(\ref{condition1a})$, $(\ref{a2m1a2mp2})$, and $(\ref{b1ncal})$, 
\begin{equation}
b_1(n)=a_1(n) = a_2(m)+1 < a_2(m+1) \leq b_2(m+1). \label{ineqsecond2}   
\end{equation}
By Equations $(\ref{b1ncal})$, $(\ref{b2mcal})$ and Inequalities $(\ref{ineqfirst})$, $(\ref{ineqsecond})$, and 
$(\ref{ineqsecond2})$, we obtain $\mathrm{(i)}$.  \\
\noindent {\tt Case 2}:  If $g(n)=0$, then by Equation $(\ref{gngmeq2})$, $g(m)=1$.
Hence,
\begin{equation}
b_1(n)=a_1(n) + g(n)-1=a_1(n)-1= \lfloor n \phi \rfloor -1,\label{inequa1}
\end{equation}
and by Lemma \ref{compareab} and Equation (\ref{condition1a}), we obtain 
\begin{equation}
b_2(m)=a_2(m)+g(m)=a_2(m)+1= a_1(n)= \lfloor n \phi \rfloor \label{inequa2}   
\end{equation}
and 
\begin{equation}
b_1(n-1) \leq a_1(n-1)<a_1(n)-1 = a_1(n)+g(n)-1= b_1(n).\label{inequa3}
\end{equation}
By Equations $(\ref{inequa1})$, $(\ref{inequa2})$, and Inequality $(\ref{inequa3})$, we obtain 
\begin{equation}
b_1(n-1) < b_1(n) < b_2(m).\label{abn1b1nb2m}  
\end{equation}
By Equations $(\ref{condition1a})$, $(\ref{inequa1})$, and Inequalities $(\ref{aam1a2m2})$, and $(\ref{abn1b1nb2m})$, we obtain
\begin{equation}
b_2(m-1) \leq a_2(m-1)+1  < a_2(m)  = a_1(n)-1 = b_1(n)<b_2(m).   \label{b2m1b1nb2m}   
\end{equation}
By Equations $(\ref{inequa1})$, $(\ref{inequa2})$, Inequalities $(\ref{abn1b1nb2m})$, $(\ref{b2m1b1nb2m})$,
we obtain $\mathrm{(ii)}$.\\
\end{proof}

\begin{lemma}\label{fromntoplemma2}
Let $n,m \in \mathbb{N}$ such that 
\begin{equation}
a_1(n)  =a_1(n-1) + 1 = a_2(m)+2. \label{condition2a}
\end{equation}
Then, one of the following statements is true:\\
$(i)$  $\lfloor n \phi \rfloor  = b_1(n) \in B_1$,
$\lfloor n \phi \rfloor -1 = b_1(n-1) \in B_1$ and 
$b_2(m) < b_1(n-1) < b_1(n) < b_2(m+1)$;\\
$(ii)$  $\lfloor n \phi \rfloor  = b_1(n) \in B_1$,
$\lfloor n \phi \rfloor -1 = b_2(m) \in B_2$ and
$b_1(n-1) < b_2(m) < b_1(n) < b_2(m+1)$.
\end{lemma}
\begin{proof}
By Definition \ref{newfunction} and Equation $(\ref{condition2a})$, we obtain
\begin{equation}
g(n) = 1  \label{gnequal1}
\end{equation}
and 
\begin{equation}
g(n-1)=1-g(m). \label{gn1euqal1gm}   
\end{equation}
Then,
\begin{equation}
b_1(n)=a_1(n)+g(n)-1=a_1(n)=\lfloor n \phi \rfloor. \label{equatio30} 
\end{equation}
Since $\lfloor m (\phi +1) \rfloor$, $\lfloor (n-1) \phi \rfloor$, and $\lfloor n \phi \rfloor$ are consecutive numbers, by Lemma \ref{rayleigh},
$\lfloor (m+1) (\phi +1) \rfloor > \lfloor n \phi \rfloor$. Hence, by Lemma \ref{compareab} and Equation $(\ref{equatio30})$,
\begin{equation}
b_2(m+1) \geq a_2(m+1) > a_1(n)=b_1(n).\label{equatio4}
\end{equation}
We have two cases.\\
\noindent {\tt Case 1}: If $g(n-1)=1$, then by Equation $(\ref{gn1euqal1gm})$, we obtain  $g(m)=0$. Hence, we obtain 
\begin{equation}
b_2(m)=a_2(m)+g(m)=a_2(m),\label{equatio1}
\end{equation}
and by Equation $(\ref{condition2a})$,
\begin{equation}
b_1(n-1)=a_1(n-1)+g(n-1)-1=a_1(n-1)=\lfloor n \phi \rfloor-1. \label{equatio2}
\end{equation}
By Equations $(\ref{condition2a})$, $(\ref{equatio1})$, and $(\ref{equatio2})$,
\begin{equation}
b_2(m) = a_2(m) = a_1(n-1)-1 = b_1(n-1)-1 < b_1(n-1),\label{inequalityb2b1a}
\end{equation}
and by Equations $(\ref{condition2a})$, $(\ref{equatio30})$, $(\ref{equatio2})$ and Inequality $(\ref{equatio4})$
, we obtain 
\begin{equation}
b_1(n-1) = a_1(n-1) = a_1(n)-1 = b_1(n)-1 < b_1(n) < b_2(m+1). \label{inequalityb2b1b}
\end{equation}
By Equations $(\ref{equatio30})$, $(\ref{equatio2})$, and Inequalities $(\ref{inequalityb2b1a})$ and $(\ref{inequalityb2b1b})$
, we obtain $\mathrm{(i)}$.
\noindent {\tt Case 2}: If $g(n-1)=0$, then by Equation $(\ref{gn1euqal1gm})$, $g(m)=1$. Hence, 
\begin{equation}
b_2(m)=a_2(m)+g(m)=a_2(m)+1=a_1(n-1)=\lfloor n \phi \rfloor-1 \label{equatio10}  
\end{equation}
and 
\begin{equation}
b_1(n-1)=a_1(n-1)+g(n-1)-1 = a_1(n-1)-1. \label{equatio20} 
\end{equation}
By Equations $(\ref{equatio10})$ and $(\ref{equatio20})$,
\begin{equation}
b_1(n-1) = a_1(n-1)-1 < b_2(m)=a_1(n-1). \label{a1n1b2mb1nb2m1in}    
\end{equation}
and
by Equations $(\ref{condition2a})$, $(\ref{equatio30})$, and  Inequality $(\ref{equatio4})$, 
\begin{equation}
a_1(n-1) = a_1(n)-1 = b_1(n)-1 < b_1(n) < b_2(m+1). \label{a1n1b2mb1nb2m1in2}    
\end{equation}
By Equations $(\ref{equatio30})$, $(\ref{equatio10})$ and Inequalities  $(\ref{a1n1b2mb1nb2m1in})$, and $(\ref{a1n1b2mb1nb2m1in2})$, we obtain 
$\mathrm{(ii)}$.
\end{proof}  

\begin{lemma}\label{b1b2lemma}
For the sets $B_1$ and $B_2$, we obtain the following:  \\
$(i)$ for $n \in   \mathbb{N}$,
\begin{equation}
b_2(n-1) < b_2(n);  \label{b1b1big}
\end{equation}
$(ii)$ for $n \in   \mathbb{N}$ with $n \geq 2$,
\begin{equation}
b_1(n-1) < b_1(n);  \label{b1b1big2}
\end{equation}
$(iii)$
\begin{equation}
B_1 \cup B_2 = \mathbb{Z}_{\geq0};  \label{b1b2eqz}
\end{equation}
$(iv)$
\begin{equation}
B_1 \cap B_2 = \emptyset.  \label{b1b2empty} 
\end{equation}
\end{lemma}
\begin{proof}
For $n \in   \mathbb{N}$, Inequality $(\ref{b1b1big})$ follows directly from Definition \ref{definitionofb1b2} and $\mathrm{(iii)}$ of Lemma \ref{lemmaphi12}. For any $x \in \mathbb{Z}_{\geq0}$, by Lemma \ref{nn1eqx}, there exists $n \in \mathbb{Z}_{\geq0}$ such that $x = \lfloor n \phi \rfloor$ or $x = \lfloor n \phi \rfloor -1$. By Lemmas \ref{fromntoplemma1} and \ref{fromntoplemma2},  $\{ \lfloor n \phi \rfloor, \lfloor n \phi \rfloor-1\} \subset B_1 \cup B_2$. Hence, $x \in B_1 \cup B_2$, and we obtain Equation $(\ref{b1b2eqz})$. Next we prove Inequality $(\ref{b1b1big2})$ and Equation $(\ref{b1b2empty})$.  By Example \ref{exampleforb1b2}, we have  $b_1(0)=b_1(1)=0$, $b_1(2)=3$, $b_2(0)=1$, $b_2(1)=2$. Hence, it is sufficient to study the case where $n \geq 2$.  By $\mathrm{(i)}$ of Lemma \ref{lemmaphi12}, we have two cases.\\
\noindent {\tt Case 1}: Suppose that $\lfloor n \phi \rfloor = \lfloor (n-1) \phi  \rfloor +2 $. Then, by Lemma \ref{rayleigh}, we obtain 
\begin{equation}
\lfloor m \phi \rfloor = \lfloor n \phi  \rfloor +1 \nonumber
\end{equation}
for some $m$, and by Lemma \ref{fromntoplemma1} we have two subcases.\\
\noindent {\tt Subcase 1}:
Suppose that for  $n,m \in \mathbb{N}$ $b_1(n-1) < b_2(m) < b_1(n) < b_2(m+1)$. Then, we obtain
Inequality $(\ref{b1b1big2})$. By $b_2(m) < b_1(n) < b_2(m+1)$, we obtain Equation $(\ref{b1b2empty})$.\\
\noindent {\tt Subcase 2}:
We suppose that $b_1(n-1) < b_1(n) < b_2(m)$, and $b_2(m-1) < b_1(n) < b_2(m)$.
We obtain Inequality $(\ref{b1b1big2})$ and Equation $(\ref{b1b2empty})$.\\
\noindent {\tt Case 2}: Suppose that $\lfloor n \phi \rfloor = \lfloor (n-1) \phi  \rfloor +1 $. Then, we obtain 
$\lfloor m \phi \rfloor + 2= \lfloor n \phi  \rfloor $ for some $m$, and by Lemma \ref{fromntoplemma2} we have two subcases.\\
\noindent {\tt Subcase 1}:
Suppose that $b_2(m) < b_1(n-1) < b_1(n) < b_2(m+1)$. Then, we obtain Inequality $(\ref{b1b1big2})$ and Equation $(\ref{b1b2empty})$.\\
\noindent {\tt Subcase 2}:
Suppose that $b_1(n-1) < b_2(m) < b_1(n) < b_2(m+1)$. Then, we obtain Inequality $(\ref{b1b1big2})$ and Equation $(\ref{b1b2empty})$.
\end{proof}

\begin{lemma}\label{frompnotp}
If we start with a position $(x,y) \in P_1$,
we cannot reach a position in $P_1$.
\end{lemma}
\begin{proof} 
Since $\{(x,y):x+y\leq 2\}$ is the set of terminal positions,
$(b_1(0),b_2(0))=(0,1)$, $(b_2(0),b_1(0))=(1,0)$, $(b_1(1),b_2(1))=(0,2)$, $(b_2(1),b_1(1))=(2,0)$ are terminal positions.
Hence it is sufficient to prove for $(x,y) = (b_1(n), b_2(n)) \in P_1$ with $n \geq 2$.
Let $m \in \mathbb{Z}_{\geq0}$ such that $m < n$.
For $(s,t) = (b_1(m), b_2(m))$ and $(u,v) = (b_2(m), b_1(m))$, by Lemma \ref{b1b2lemma}, we have
$b_1(n) \ne b_1(m), b_2(m)$ and $b_2(n) \ne b_1(m), b_2(m)$. Hence $M_i(x,y) \cap \{(u,v),(s,t)\} = \emptyset$ for $i=1,2$.
 $b_1(n)-b_2(n)=-n-1$, $b_1(m)-b_2(m)=-m-1$, and 
$b_2(m)-b_1(m)=m+1$, and hence 
$M_3(x,y) \cap \{(u,v),(s,t)\} = \emptyset$.
\end{proof}

Next, we aim to prove that from any position $(x,y) \notin P_1$,
we can reach a position in $P_1$. We need some lemmas.

\begin{lemma}\label{formptopg} If you can move to a position in $P_0$ by the diagonal move from a position that is not in $P_1$, you can move to a position in $P_1$ by the diagonal move. 
\end{lemma} 
\begin{proof} If you can move from $(x,y) \notin P_1$ to $(0,0)$, then you can move to $(1,1) \in P_1$. It is sufficient to prove that you can move to $(u,v) = (\lfloor (n+1) \phi  \rfloor + n +1, \lfloor (n+1) \phi  \rfloor) \in P_0$. Since  $(\lfloor n (\phi + 1) \rfloor +g(n),\lfloor n \phi \rfloor+g(n)-1) \in P_1$, \begin{equation} (\lfloor n (\phi + 1) \rfloor +1, \lfloor n \phi \rfloor) \text{ or } (\lfloor n (\phi + 1) \rfloor , \lfloor n \phi \rfloor -1) \in P_1. \label{g0org1} \end{equation} Since \begin{equation} \{(\lfloor n (\phi + 1) \rfloor +1, \lfloor n \phi \rfloor), (\lfloor n (\phi + 1) \rfloor , \lfloor n \phi \rfloor -1) \}\subset   M_3(u,v) \subset M_3(x,y),\nonumber \end{equation} by $(\ref{g0org1})$, we finish the proof. 
\end{proof}

Now we begin to prove that from a position not in $P_1$ we can move to a position in $P_1$. A relatively complicated proof is required because $M_1$, $M_2$ or $M_3$ must be chosen according to the starting position $(x,y)$. Let's me illustrate the method of proof using Example \ref{examfortwofigures}.

\begin{example}\label{examfortwofigures}
In Figures~\ref{graph4}, the positions denoted by red squares belong to $P_0$, and those denoted by light blue squares belong to $P_1$.
We prove that from a position that is not printed in light blue, we can move to a
position printed in light blue.

First, we study the positions in the blue area, and this area is
between two positions $(\lfloor 2 \phi  \rfloor, \lfloor 2 \phi \rfloor + 2)=(3,5)\in P_0$
and $(\lfloor 4 \phi  \rfloor + 4, \lfloor 4 \phi \rfloor)=(10,6)\in P_0$.

We denote positions in the blue area by 
$(x,y)=(\lfloor 2 \phi  \rfloor + k, \lfloor 2 \phi \rfloor + 2+u)=(3 + k, 5+u)$ 
for $k = 1,2,3,4,5,6$ and $u=0,1$.

We have two cases. In the first case, we suppose that $1 \leq k \leq 2+u$. From the point $(x,y)=(3 + k, 5+u)$, we can move to the point $(v,w)=(\lfloor (2+u-k) \phi   \rfloor , \lfloor (2+u-k) (\phi +1)  \rfloor  \in P_{0,1}$ by the diagonal move $M_3$, because $y-x=2+u-k=w-v$. 
Then, passing through this position in $P_0$, by Lemma \ref{formptopg}, we can move to a position in $P_{1,1}$ by the diagonal move $M_3$.

In the second case, we suppose that $2+u < k \leq 6$. We have two subcases.

In the first subcase, we suppose that $u=0$ and $k=6$. 
From the point $(x,y) = (3+k,5+u)=(9,5)$, we can move to $(u,v)=(8,4) \in P_{1,2}$ by the diagonal move without passing through any position in $P_0$.

In the second subcase, we suppose that $u=1$ or $k<6$.
From the position $(x,y)=(3 + k, 5 +u)$,
we can move to $(v,w)=(\lfloor (k-2-u) (\phi + 1)  \rfloor , \lfloor (k-2-u) \phi   \rfloor ) \in P_{0,2}$ by the diagonal move $M_3$, because $y-x=u+2-k = w-v$.
Then, by Lemma \ref{formptopg}, we can move to a position in $P_{1}$ by the diagonal move $M_3$. \\

Next, we discuss the area to the left of the blue area. Here, we denote the positions in the area 
as $(3 - k, 5 +u)$ with $k=0,1,2,3$ and $u=0,1$.
We have a red square, a light blue square and yellow squares. A red square is a position in $P_0$, and hence, by Lemma \ref{formptopg}, we can move to a position in $P_{1,1}$ by the diagonal move $M_3$. We need not consider the light blue square, which is a position in $P_1$. From one of the positions denoted by yellow squares, we can reach a position in $P_1$ by the vertical move $M_2$, but it depends on the value of $3 - k$ that we can move to a position in $P_{1,1}$ or $P_{1,2}$.   On the right side of the blue area, we denote the positions as $(10 + k, 6 -u)$ with $k=0,1,2,3, \dots $ and $u=0,1$. We have a red square, a light blue square and green squares. A red square is a position in $P_0$, and by Lemma \ref{formptopg}, we can move to a position in $P_{1,2}$ by the diagonal move $M_3$. We need not consider the light blue square that is a position in $P_1$. From one of the positions denoted by green squares, we can reach $(10,5) \in P_{1,2}$ or $(3,6) \in P_{1,1}$ by the horizontal move $M_2$.
\end{example}

In Example \ref{examfortwofigures}, we used 
Figure \ref{graph4}, but in Lemma \ref{fromntoplemma1b},
we have the situation whose example can be described by Figure 
\ref{graph1}.

\begin{figure}[H]
\includegraphics[height=2.7cm]{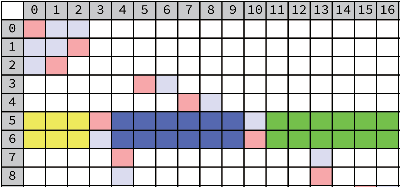}
\caption{ \ }\label{graph4}
\end{figure}

\begin{figure}[H]
\includegraphics[height=3.9cm]{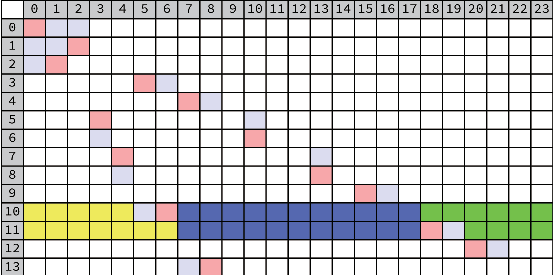}
\caption{ \ }\label{graph1}
\end{figure}

\begin{lemma}\label{fromntoplemma1b}
Let $n,m \in \mathbb{N}$ such that 
\begin{equation}
\lfloor n \phi \rfloor = \lfloor m (\phi +1)  \rfloor  +1. \label{condition1}
\end{equation}
Suppose that $(x,y) \notin P_1$ and $y= \lfloor n \phi \rfloor $ or 
$y= \lfloor n \phi \rfloor -1$.
Then $\textit{move}(x,y) \cap P_1 \ne \emptyset$.
\end{lemma}
\begin{proof}
By Equation $(\ref{condition1})$ and Definition \ref{newfunction}, 
\begin{equation}
g(n) = 1 - g(m).  \label{gngmeq}  
\end{equation}
$\mathrm{(i)}$ Suppose that  
\begin{equation}
(x,y)=(\lfloor m \phi  \rfloor + k, \lfloor m (\phi +1) \rfloor +u) \notin P_1   \nonumber
\end{equation}
such that
$1 \leq k \leq m+n$ and $u=0,1$. Note that $(x,y)$ is on the right side of the point
$(\lfloor m \phi \rfloor, \lfloor m (\phi +1)  \rfloor ) \in P_{0,1}$.
We prove that we can move to a position in $P_1$. In Figures \ref{graph4} or \ref{graph1}, this position $(x,y)$
 belongs to the area printed in blue.\\
We have two cases.\\
\noindent {\tt Case 1}: Suppose that $1 \leq k \leq m+u$. Then, $m+u-k \geq 0$, and from the point $(x,y)=(\lfloor  m \phi  \rfloor + k, \lfloor m ( \phi +1) \rfloor +u)$, we can move to the point $(u,v)=(\lfloor (m+u-k) \phi   \rfloor , \lfloor (m+u-k) (\phi +1)  \rfloor ) \in P_{0,1}$ by the diagonal move $M_3$, because
$y-x=m+u-k=v-u$. Then, by Lemma \ref{formptopg}, we can move to a position in $P_{1,1}$ by the diagonal move $M_3$.\\
\noindent {\tt Case 2}: Suppose that $m+u < k \leq m+n$. We have two subcases.\\
\noindent {\tt Subcase 1}: Suppose that $u=0$ and $k=m+n$. Then, by $(\ref{condition1})$,
from the position 
\begin{align}
(x,y) & =  (\lfloor  m \phi   \rfloor +m+n, \lfloor m \phi \rfloor +m) \nonumber \\
&  = (\lfloor  m (\phi + 1)  \rfloor +n, \lfloor m \phi \rfloor +m)  \nonumber \\
& = (\lfloor  n (\phi + 1)  \rfloor -1, \lfloor n \phi \rfloor -1), \nonumber 
\end{align}
we can move to 
\begin{equation}
(u,v)=(\lfloor  (n-1) (\phi + 1)  \rfloor +g(n-1), \lfloor (n-1) \phi \rfloor + g(n-1)-1) \in P_{1,2} \nonumber    
\end{equation}
 by the diagonal move $M_3$, because $x-y=n=u-v$.\\
\noindent {\tt Subcase 2}: Suppose that $u=1$ or $k<m+n$.
From the position 
\begin{equation}
(\lfloor  m \phi  \rfloor + k, \lfloor m (\phi + 1) \rfloor +u), \nonumber  
\end{equation}
we can move to 
\begin{equation}
(\lfloor (k-m-u) (\phi + 1)  \rfloor , \lfloor (k-m-u) \phi   \rfloor ) \in P_{0,2}  \nonumber  
\end{equation}
 by the diagonal move $M_3$.
Then, by Lemma \ref{formptopg}, we can move to a position in $P_{1}$ by the diagonal move $M_3$. \\
$\mathrm{(ii)}$   Suppose that 
\begin{equation} 
(x,y)=(\lfloor m \phi  \rfloor -k, \lfloor m (\phi +1) \rfloor +u) \notin P_1 \nonumber   
\end{equation} 
with $u=0,1$ and $k \in \mathbb{Z}_{\geq0}$.  If we use Example \ref{examfortwofigures} with Figure \ref{graph4}, the position $(x,y)$ belongs to the area that is on the left of the blue area. Here, we have a red square, a light blue square, and many yellow squares. We will prove that  we can move to a position in $ P_1$ by the vertical move $M_3$.\\
We have three cases.\\
\noindent {\tt Case 1}: Suppose that $k=u=0$. Then, $(x,y) \in P_{0,1}$ and by Lemma \ref{formptopg}, we can move to a position in $P_1$ by the diagonal move $M_3$.\\
\noindent {\tt Case 2}: Suppose that $k=0$ and $u=1$. If $g(m)=1$, 
\begin{equation}
(x,y)=(\lfloor m \phi  \rfloor + g(m)-1, \lfloor m (\phi + 1) \rfloor +g(m)) \in P_{1,1}.\nonumber
\end{equation}
 This contradicts the assumption that $(x,y) \notin P_1$. Hence, $g(m)=0$, and we can move from $(\lfloor m \phi  \rfloor  , \lfloor m (\phi + 1) \rfloor +1)$
 to 
\begin{align}
& (\lfloor m \phi  \rfloor -1 , \lfloor m (\phi + 1) \rfloor) \nonumber \\
   & =(\lfloor m \phi  \rfloor + g(m)-1 , \lfloor m (\phi + 1) \rfloor +g(m)) \in P_{1,1}  \nonumber 
\end{align}
 by the diagonal move.\\
\noindent {\tt Case 3}: Suppose that $k \geq 1$.  It depends on the value of the first coordinate that you can move to a position in $P_{1,1}$ or $P_{1,2}$. Hence, there are two subcases.\\
\noindent {\tt Subcase 1}: Suppose that there exists $t  \in \mathbb{Z}_{\geq0}$ such that
\begin{equation}
\lfloor m \phi  \rfloor -k = \lfloor t \phi  \rfloor +g(t)-1.\nonumber
\end{equation}
If $t < m$, then we can move by the vertical move $M_2$ to the position 
\begin{align}
& (\lfloor m \phi  \rfloor -k, \lfloor t (\phi + 1) \rfloor  + g(t)) \nonumber \\
& =(\lfloor t \phi  \rfloor +g(t)-1, \lfloor t (\phi + 1) \rfloor  + g(t)) \in P_{1,1}. \nonumber
\end{align}
If $t=m$, then $k=1$ and $g(t)=0$. Hence,
\begin{align}
(x,y) & =(\lfloor m \phi  \rfloor -1,\lfloor m (\phi+1) \rfloor +u) \nonumber \\
& = (\lfloor t \phi  \rfloor + g(t)-1,\lfloor t (\phi+1) \rfloor +g(t)+u). \nonumber
\end{align}
If $u=0$, $(x,y) \in P_{1,1}$. This contradicts the assumption of the lemma. Hence, $u=1$. Then, we can move to the position $(x,y-1)\in P_{1,1}$ by the vertical move $M_2$.\\
\noindent {\tt Subcase 2}: Suppose that there exists $t  \in \mathbb{Z}_{\geq0}$ such that
\begin{equation}
\lfloor m \phi  \rfloor -k = \lfloor t (\phi + 1)  \rfloor +g(t).\nonumber
\end{equation}
Then, we have $t<m$, and we can move by the vertical move $M_2$ to the position
\begin{align}
& (\lfloor m \phi  \rfloor -k, \lfloor t \phi \rfloor + g(t)-1) \nonumber \\
& =(\lfloor t (\phi +1)  \rfloor +g(t), \lfloor t \phi \rfloor + g(t)-1) \in P_{1,2}. \nonumber
\end{align}
$\mathrm{(iii)}$  Suppose that 
\begin{equation}
(x,y)=(\lfloor n (\phi + 1)  \rfloor  + k, \lfloor n \phi \rfloor -u) \notin P_1 \nonumber
\end{equation}
with $u=0,1$ and $k \in \mathbb{Z}_{\geq0}$. If we use Example \ref{examfortwofigures} with Figure \ref{graph4}, the position $(x,y)$ 
belongs to the area to the right of the blue area. Here, we have a red square, a light blue square,
and many green squares.
We will prove that we can move to a position in $P_1$ by the horizontal move.
We have two cases.\\
\noindent {\tt Case 1}: Suppose that $g(n)=0$ and $g(m)=1$. An example of this situation is presented in Figure \ref{graph4}.
If $k=0$ and $u=1$, 
$(x,y)=(\lfloor n (\phi + 1)  \rfloor , \lfloor n \phi \rfloor -1)=(\lfloor n (\phi + 1)  \rfloor +g(n), \lfloor n \phi \rfloor +g(n) -1) \in P_{1,2}$. Hence, we suppose that $k>0$ or $u=0$. We have two subcases.\\
\noindent {\tt Subcase 1}: Suppose that $u=0$.
Then, from the position 
\begin{equation}
(x,y)=(\lfloor n (\phi + 1)  \rfloor  + k, \lfloor n \phi \rfloor ),\nonumber    
\end{equation}
 by $(\ref{condition1})$, we can move to
\begin{align}
& (\lfloor m \phi  \rfloor , \lfloor m (\phi + 1)  \rfloor +1 ) \nonumber \\
& = (\lfloor m \phi  \rfloor +g(m)-1, \lfloor m (\phi + 1)  \rfloor +g(m)) \in P_{1,1}. \nonumber
\end{align}
\noindent {\tt Subcase 2}: Suppose that $k >0$. Then, 
from the position 
\begin{equation}
(x,y)=(\lfloor n (\phi + 1)  \rfloor  + k, \lfloor n \phi \rfloor -1 ),\nonumber
\end{equation}
we can move to
\begin{equation}
(\lfloor n (\phi + 1)  \rfloor , \lfloor n \phi  \rfloor -1) \in P_{1,2}.\nonumber
\end{equation}
\noindent {\tt Case 2}: If $g(n)=1$, then $g(m)=0$.  An example of this situation is presented in Figure \ref{graph1}. We have two subcases.\\
\noindent {\tt Subcase 1}:  Suppose that $u=0$.
If $k \geq 2$,
from the position $(x,y)=(\lfloor n (\phi + 1)  \rfloor + k, \lfloor n \phi \rfloor )$, we can move to
\begin{equation}
(\lfloor n (\phi + 1)  \rfloor  +1, \lfloor n \phi  \rfloor)=(\lfloor n (\phi + 1)  \rfloor +g(n), \lfloor n \phi  \rfloor +g(n)-1) \in P_{1,2}\nonumber
\end{equation}
by the horizontal move $M_1$.

If $k=1$,  
\begin{align}
(x,y) & =(\lfloor n (\phi + 1)  \rfloor  +1, \lfloor n \phi  \rfloor) \nonumber \\
&  =(\lfloor n (\phi + 1)  \rfloor  +g(n), \lfloor n \phi  \rfloor +g(n)-1) \in P_{1,2}, \nonumber
\end{align}
and this contradicts the assumption of the lemma.

If $k=0$, $(x,y) \in P_{0,2}$. Then, by Lemma \ref{formptopg}, we can move to a position in $P_1$.
\noindent {\tt Subcase 2}:  Suppose that $u=1$.
By $(\ref{condition1})$, $\lfloor n \phi  \rfloor -1 =\lfloor m (\phi+1)  \rfloor$. Hence, 
from the position $(x,y)=(\lfloor n (\phi + 1)  \rfloor  + k, \lfloor n \phi \rfloor - 1)$, we can move to
\begin{equation}
(\lfloor m \phi  \rfloor -1, \lfloor m (\phi + 1)  \rfloor  = (\lfloor m \phi  \rfloor +g(m)-1, \lfloor m (\phi + 1)  \rfloor +g(m)) \in P_{1,1}\nonumber
\end{equation}
by the horizontal move $M_1$.
\end{proof}  

\begin{example}\label{examfortwofigures2}
In Figures~\ref{graph2} and \ref{graph0}, the positions of $P_1$ are printed in light blue. 
Positions belonging to $P_0$ are printed in red, and from one of these positions, by Lemma \ref{formptopg}, we can move to a position in $P_1$.
From one of the positions denoted by yellow squares, we can reach a position in $P_1$ by the vertical move $M_2$.
From one of the positions denoted by blue squares, we can pass through a position in $P_0$, and reach a position in $P_1$ by the diagonal move $M_3$, and from one of the positions denoted by green squares, we can reach a position in $P_1$ by the horizontal move $M_2$.

These moves are examples of the methods used in the proof of Lemma  \ref{fromntoplemma2b}.
\end{example}

\begin{figure}[H]
\includegraphics[height=4.5cm]{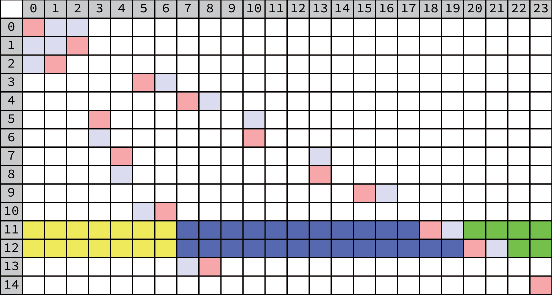}
\caption{ \ }\label{graph2}
\end{figure}

\begin{figure}[H]
\includegraphics[height=3.3cm]{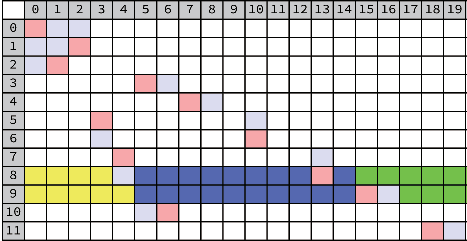}
\caption{ \ }\label{graph0}
\end{figure}

\begin{lemma}\label{fromntoplemma2b}
Let $n,m \in \mathbb{N}$ such that 
\begin{equation}
\lfloor n \phi \rfloor = \lfloor (n-1) \phi \rfloor +1 =\lfloor m (\phi + 1) \rfloor  +2. \label{condition2}
\end{equation}
Suppose that $(x,y) \notin P_1$ and $y= \lfloor n \phi \rfloor $ or 
$y= \lfloor n \phi \rfloor -1$.
Then $\textit{move}(x,y) \cap P_1 \ne \emptyset.$
\end{lemma}
\begin{proof}
By Definition \ref{newfunction} and Equation (\ref{condition2}), we obtain
\begin{equation}
g(n) = 1 \label{gn1valuse} 
\end{equation}
and 
\begin{equation}
g(n-1)=1-g(m). \label{gn1eq1gm}  
\end{equation}
$\mathrm{(i)}$ Suppose that  
\begin{equation}
(x,y)=(\lfloor  m \phi  \rfloor + k, \lfloor  m (\phi + 1) \rfloor +u) \notin P_1 \nonumber   
\end{equation}
such that $1 \leq k \leq m+n+1$ and $u=1,2$. If we use Example \ref{examfortwofigures2} with Figures \ref{graph2} and \ref{graph0}, the point $(x,y)$ lies in the blue area. We prove that 
we can move to a position in $P_1$. We have three cases.\\
\noindent {\tt Case 1}: Suppose that $u=2$ and $k = 1$. From $(x,y)=
(\lfloor  m \phi  \rfloor+1,\lfloor  m (\phi +1)  \rfloor + 2)$, we can move to 
$(x,y)=(\lfloor  m \phi  \rfloor + g(m)-1, \lfloor  m (\phi + 1) \rfloor +g(m)) \in P_{1,1}$ by the diagonal move $M_3$.\\
\noindent {\tt Case 2}: Suppose that $1 \leq k \leq m+u$ and $k \geq u$. Then, from a position  $(\lfloor  m \phi  \rfloor + k, \lfloor  m (\phi +1) \rfloor +u)$, 
we can move to $(\lfloor (m+u-k) \phi   \rfloor , \lfloor (m+u-k) (\phi + 1)  \rfloor  \in P_0$ by the diagonal move $M_3$. Then, by Lemma \ref{formptopg}, we can move to a position in $P_1$ by the diagonal move $M_3$.\\
\noindent {\tt Case 3}: Suppose that $m+u < k \leq m+n+1$. We have three subcases.\\
\noindent {\tt Subcase 1}: If $k=m+n+1$ and $u=1$.
Then, by Equation $(\ref{condition2})$ 
\begin{align}
(x,y) & =(\lfloor  m \phi  \rfloor + m+n+1, \lfloor  m (\phi +1)  \rfloor +1) \nonumber \\
&  =(\lfloor  (n-1) (\phi +1)  \rfloor +1, \lfloor  (n-1) \phi  \rfloor ). \nonumber
\end{align}
Then, if $g(n-1)=1$,
\begin{equation}
(x,y) =(\lfloor  (n-1) (\phi +1)  \rfloor +g(n-1), 
\lfloor  (n-1) \phi  \rfloor +g(n-1)-1) \in P_1, \nonumber
\end{equation}
and this contradicts the assumption of the lemma.

If $g(n-1)=0$, from $(x,y)$, we can move to the position
\begin{align}
& (\lfloor  (n-1) (\phi +1)  \rfloor, \lfloor  (n-1) \phi  \rfloor -1) \nonumber \\
& = (\lfloor  (n-1) (\phi +1)  \rfloor +g(n-1), \lfloor  (n-1) \phi  \rfloor +g(n-1)-1) \in P_1 \nonumber
\end{align}
by the diagonal move $M_3$.\\
\noindent {\tt Subcase 2}: Suppose that $m+u < k<m+n$ or $u=2$ and $m+u < k \leq m+n+1$.
Then, from a position $(\lfloor  m \phi  \rfloor + k, \lfloor  m (\phi + 1) \rfloor  +u)$, 
we can move to $(\lfloor (k-m-u) (\phi + 1)   \rfloor  , \lfloor (k-m-u) \phi   \rfloor ) \in P_0$ by the  diagonal move.
Then, by Lemma \ref{formptopg}, we can move to a position in $P_1$\\
\noindent {\tt Subcase 3}: If $k=m+n$ and $u=1$, by Equation $(\ref{condition2})$
\begin{align}
(x,y)& =(\lfloor  m \phi  \rfloor + m + n, \lfloor  m (\phi +1)  \rfloor +1) \nonumber \\
& = (\lfloor  (n-1) (\phi +1) \rfloor, \lfloor  (n-1) \phi  \rfloor) \in P_0. \nonumber
\end{align}
Then, by Lemma \ref{formptopg}, we can move to a position in $P_1$.\\
$\mathrm{(ii)}$  Suppose that 
\begin{equation}
(x,y)=(\lfloor m \phi  \rfloor -k, \lfloor m (\phi + 1) \rfloor +u) \notin P_1 \nonumber   
\end{equation}
with $u=1,2$ and $k \in \mathbb{Z}_{\geq0}$. If we use the Example \ref{examfortwofigures2} with Figures \ref{graph2} and \ref{graph0}, the point $(x,y)$ lies in the yellow area. We prove that 
if $(x,y) \notin P_1$ we can move to a position in 
$P_1$ by the vertical move.  It depends on the value of the first coordinate that you can move to a position in $P_{1,1}$ or $P_{1,2}$. Hence, there are two cases.\\
\noindent {\tt Case 1}: Suppose that there exists $t  \in \mathbb{Z}_{\geq0}$ such that
\begin{equation}
\lfloor m \phi  \rfloor -k = \lfloor t \phi  \rfloor +g(t)-1.\nonumber
\end{equation}
Then, there are three subcases.\\
\noindent {\tt Subcase 1}:
If $t < m$, then we can move to the position 
\begin{align}
& (\lfloor m \phi  \rfloor -k, \lfloor t (\phi + 1) \rfloor  + g(t)) \nonumber \\
& =(\lfloor t \phi  \rfloor +g(t)-1, \lfloor t (\phi + 1) \rfloor  + g(t)) \in P_{1,1} \nonumber
\end{align}
by the vertical move $M_2$.\\
\noindent {\tt Subcase 2}:
If $t=m$ and $k=1$, then $g(t)=0$. Hence,
\begin{align}
(x,y) & =(\lfloor m \phi  \rfloor -1,\lfloor m (\phi+1) \rfloor +u) \nonumber \\
& = (\lfloor t \phi  \rfloor + g(t)-1,\lfloor t (\phi+1) \rfloor +g(t)+u). \nonumber
\end{align}
Then, we can move to the position $(x,y-u)\in P_{1,1}$.\\
\noindent {\tt Subcase 3}:
Suppose that $t=m$ and $k=0$, then $g(t)=1$. Hence, 
\begin{equation}
(x,y)= (\lfloor t \phi  \rfloor + g(t)-1,\lfloor t (\phi+1) \rfloor +g(t)+u-1) \nonumber
\end{equation}
 If $u=1$, then $(x,y) \in P_1$. If $u=2$, we can move to the position $(x,y-1)\in P_{1,1}$ by the vertical move $M_2$.\\
\noindent {\tt Case 2}: Suppose that there exists $t  \in \mathbb{Z}_{\geq0}$ such that
\begin{equation}
\lfloor m \phi  \rfloor -k = \lfloor t (\phi + 1)  \rfloor +g(t).\nonumber
\end{equation}
Then, we have  $t<m$ and we can move by the vertical move $M_2$ to the position
\begin{align}
&  (\lfloor m \phi  \rfloor -k, \lfloor t \phi \rfloor + g(t)-1) \nonumber \\
& =(\lfloor t \phi  \rfloor +g(t)-1, \lfloor t \phi \rfloor + g(t)-1) \in  P_{1,2}. \nonumber
\end{align}
$\mathrm{(iii)}$  Suppose that 
\begin{equation}
(x,y)=(\lfloor n (\phi +1) \rfloor + k, \lfloor n \phi \rfloor -u) \notin P_1 \nonumber   
\end{equation}
with $u=0,1$ and $k \in \mathbb{Z}_{\geq0}$. We prove that we can move to a position in $P_1$ by the horizontal move $M_1$. If we use the Example \ref{examfortwofigures2} with Figures \ref{graph2} and \ref{graph0}, the point $(x,y)$ lies in the green area.  
By $(\ref{gn1valuse})$,
\begin{equation}
(\lfloor n (\phi + 1)  \rfloor + 1,\lfloor n \phi  \rfloor) =(\lfloor n (\phi + 1)  \rfloor  + g(n),\lfloor n \phi  \rfloor + g(n)-1) \in P_{1,2}.\label{ku0case}
\end{equation}
First, we consider the case that $u=0$. If $k=0$, then   $(x,y) \in P_{0,2}$, and by Lemma \ref{formptopg},
we can move to a position in $P_{1,2}$. If $k=1$, then by $(\ref{ku0case})$,  $(x,y) \in P_1$. This contradicts the assumption of the lemma.
If $k>1$,  we can move to $(\lfloor n (\phi + 1)  \rfloor + 1,\lfloor n \phi  \rfloor) \in P_{1,2}.$
Next, we assume that $u=1$. We have two cases.\\
\noindent {\tt Case 1}: Suppose $g(n-1)=1$. An example of this case is Figure \ref{graph2}.
From the position $(\lfloor n (\phi + 1)  \rfloor  + k, \lfloor n \phi  \rfloor -1)$, by Equation $(\ref{condition2})$ we can move 
by the horizontal move $M_1$ to
\begin{align}
&  (\lfloor n (\phi +1)  \rfloor -1, \lfloor n \phi  \rfloor -1 ) \nonumber \\
& = (\lfloor (n-1) (\phi +1)  \rfloor + 1, \lfloor (n-1) \phi  \rfloor )  \nonumber \\
& = (\lfloor (n-1) (\phi +1)  \rfloor + g(n-1), \lfloor (n-1) \phi  \rfloor +g(n-1)-1) \in P_{1,2}. \nonumber
\end{align}
\noindent {\tt Case 2}: Suppose $g(n-1)=0$. An example of this case is Figure \ref{graph0}.
Then by $(\ref{gn1eq1gm})$, $g(m)=1$.
From the position $(\lfloor n (\phi + 1)  \rfloor + k, \lfloor n \phi  \rfloor -1)$, by Equation $(\ref{condition2})$ we can move by the horizontal move $M_1$ to 
\begin{align}
&  (\lfloor n \phi \rfloor -m-2, \lfloor n \phi  \rfloor -1) \nonumber \\
& =(\lfloor m \phi  \rfloor , \lfloor m (\phi + 1)  \rfloor +1 )  \nonumber \\
& =(\lfloor m \phi  \rfloor + g(m)-1, \lfloor m (\phi + 1)  \rfloor +g(m) ) \in P_{1,1}. \nonumber
\end{align}
\end{proof} 


\begin{lemma}\label{fromntoplemma}
From any position  $(x,y) \notin P_1$, we can move to a position in $P_1$.
\end{lemma}
\begin{proof}
Let $y \in \mathbb{Z}_{\geq0}$. We suppose that $(x,y) \notin P_1$, and prove that 
 $\textit{move}(x,y) \cap P_1 \ne \emptyset$. By Lemma \ref{nn1eqx},
there exists $n \in \mathbb{Z}_{\geq0}$ such that $y=  \lfloor n \phi  \rfloor$ or $y=  \lfloor n \phi  \rfloor -1$.
Since $\{(u,v):u+v \leq 2\} \subset P_1$, $\textit{move}(x,y) \cap P_1 \ne \emptyset$ for 
any $(x,y) \notin P_1$ with $y \leq 2$. Hence, we assume that $n \geq 2$.
By Lemma \ref{twocases}, we have two cases.\\
\noindent {\tt Case 1}:  Suppose that there exists $m \in \mathbb{Z}_{\geq0}$ such that 
$\lfloor (n-1) \phi \rfloor$, $\lfloor m (\phi + 1) \rfloor $, $\lfloor n \phi \rfloor$ are three consecutive numbers. Then, by Lemma \ref{fromntoplemma1b}, $\textit{move}(x,y) \cap P_1 \ne \emptyset$.\\
\noindent {\tt Case 2}: Suppose that there exists $m \in \mathbb{Z}_{\geq0}$ such that 
$\lfloor m (\phi + 1) \rfloor $, $\lfloor (n-1) \phi \rfloor$, $\lfloor n \phi \rfloor$ are three consecutive numbers. Then, by Lemma \ref{fromntoplemma2b}, $\textit{move}(x,y) \cap P_1 \ne \emptyset$.
\end{proof}

\begin{theorem}\label{theoremforwythoffandvar}
The set of $\mathcal{P}$-positions of the variant of Wythoff's game in Definition \ref{wythoffvar} is 
$P_1$.
\end{theorem}
\begin{proof}
This theorem follows directly from Lemmas \ref{frompnotp} and \ref{fromntoplemma}.
\end{proof}

\section{The relation between our sequence and Hofstadter's G-sequence}
D. Hofstadter defined the following Hofstadter's G sequence in page 137 of \cite{escherbach}. In this section, we redefine the function $g$ in Definition 
\ref{newfunction} using  Hofstadter's G sequence.
\begin{definition}
The Hofstadter G sequence is defined as follows:
\begin{align}
& h(0)=0, \nonumber \\
& h(n)=n-h(h(n-1)) \text{ for } n \in \mathbb{N}. \nonumber 
\end{align}
\end{definition}

\begin{theorem}[\cite{curiousse}, \cite{strange}]\label{theoremsust}
Let $h$ be the Hofstadter's G-sequence. Then,
$h(n) = \lfloor \frac{n+1}{\phi}  \rfloor$.
\end{theorem}

\begin{lemma}\label{lemmaforgs}
For $n,m \in \mathbb{Z}_{\geq0}$, if 
\begin{equation}
\lfloor n \phi  \rfloor -2 \leq \lfloor m (\phi+1) \rfloor  \leq \lfloor n \phi \rfloor -1,\label{mphiine}
\end{equation}
then
$m = \lfloor \frac{n}{\phi}  \rfloor$.
\end{lemma}
\begin{proof}
By (\ref{mphiine}), we have 
\begin{equation}
 n \phi  -3 <   m (\phi + 1)   <  n \phi.\label{mphiine2}
\end{equation}
Since $\phi^2= \phi + 1$, by Inequality $(\ref{mphiine2})$,
\begin{equation}
 \frac{n}{\phi}  -\frac{3}{\phi + 1} <   m   <  \frac{n}{\phi}.\nonumber
\end{equation}
Then,
\begin{equation}
 \lfloor \frac{n}{\phi}\rfloor  -\frac{3}{\phi + 1} <   m   \leq  \lfloor \frac{n}{\phi}\rfloor.\label{mphiine4}
\end{equation}
Since $\phi +1 > 2$, by Inequality $(\ref{mphiine4})$,
$m=\lfloor \frac{n}{\phi}\rfloor$ or $\lfloor \frac{n}{\phi}\rfloor -1$.
Suppose that $m= \lfloor \frac{n}{\phi}\rfloor -1$. Then, by (\ref{mphiine}),
\begin{equation}
\lfloor n \phi \rfloor -2 <  m (\phi + 1)  = \lfloor \frac{n}{\phi}\rfloor (\phi +1) - (\phi + 1),\nonumber
\end{equation}
and hence,  
\begin{align}
2 - (\phi+1) & > \lfloor n \phi \rfloor - \lfloor \frac{n}{\phi}\rfloor (\phi+1)   \nonumber \\
& = \lfloor n \phi \rfloor - \lfloor n (\phi-1) \rfloor (\phi+1) \nonumber \\
& =  \lfloor n \phi \rfloor - (\lfloor n \phi \rfloor -n) (\phi+1) \nonumber \\
& =  n \phi + n - \lfloor n \phi \rfloor \phi     \nonumber \\
& = n \phi^2-\lfloor n \phi \rfloor \phi \geq 0 \label{mphiine6}
\end{align}
Since $(\phi+1) > 2$,
Inequality (\ref{mphiine6}) leads to a contradiction.
Therefore, $m \ne \lfloor \frac{n}{\phi}\rfloor -1$ and $m= \lfloor \frac{n}{\phi}\rfloor$.
\end{proof}

\begin{definition}\label{definitionoffuncf}
By Lemma \ref{twocases}, 
there exists an unique $m$ such that 
$\lfloor n \phi \rfloor -2 \leq \lfloor m(\phi+1) \rfloor \leq \lfloor n \phi \rfloor-1$. 
We define $f(n)= m$.
\end{definition}

\begin{lemma}\label{forntherem}
For the function $f$ in Definition \ref{definitionoffuncf}, we have the following:\\
$(i)$ for any $n$, $f(n)>f(n-1)$ if and only if there exists $m$ such that $\lfloor n \phi \rfloor = \lfloor m(\phi+1)  \rfloor +1$;\\
$(ii)$  $f(n)=h(n-1)$, where $h$ is Hofstadter's sequence.
\end{lemma}
\begin{proof}
By using Lemma \ref{twocases} for $n$, we have two cases.\\
\noindent {\tt Case 1}: Suppose that $\lfloor m (\phi + 1) \rfloor $,  $\lfloor (n-1) \phi \rfloor$, $\lfloor n \phi \rfloor$ are three consecutive numbers. Then, $f(n)=f(n-1)=m$.\\
\noindent {\tt Case 2}:  Suppose that $\lfloor (n-1) \phi \rfloor$, $\lfloor m (\phi +1) \rfloor$, $\lfloor n \phi \rfloor$ are three consecutive numbers. Next, we use Lemma \ref{twocases} for $n-1$. Then, we have two subcases.\\
\noindent {\tt Subcase 1}: If $\lfloor (n-2) \phi \rfloor$, $\lfloor (m-1) ( \phi +1) \rfloor$,$\lfloor (n-1) \phi \rfloor$, $\lfloor m( \phi+1) \rfloor$, $\lfloor n \phi \rfloor$ are five consecutive numbers, then $f(n)=m$ and $f(n-1)=m-1$.\\
\noindent {\tt Subcase 2}: If 
$\lfloor (m-1) (\phi + 1) \rfloor $, $\lfloor (n-2) \phi \rfloor$, $\lfloor (n-1) \phi \rfloor$, $\lfloor m (\phi + 1) \rfloor $, $\lfloor n \phi \rfloor$ are   
 five consecutive numbers, then, $f(n)=m$ and $f(n-1)=m-1$.

Therefore, $f(n)>f(n-1)$ if and only if there exists $m$ such that $\lfloor n \phi \rfloor = \lfloor m(\phi+1) \rfloor +1$.
By Theorem \ref{theoremsust} and Lemma \ref{lemmaforgs}, 
$f(n)= \lfloor  \frac{n}{\phi } \rfloor =  h(n-1)$.
\end{proof}

\begin{corollary}\label{relationhof}
For the function $g$ that is defined in Definition \ref{newfunction},
we have for $n \geq 2$,
\begin{equation}
g(n)=
\begin{cases}
1-g(h(n-1)) & (\mbox{ if } h(n-2) < h(n-1)), \nonumber \\ 
1 &  (else),\nonumber
\end{cases}
\end{equation} 
where $h$ is the Hofstadter G-sequence.
\end{corollary}
\begin{proof}
This is derived directly from Definition \ref{newfunction} and Lemma \ref{forntherem}.
\end{proof}

By Corollary \ref{relationhof}, we redefined the function $g(n)$ by 
 the Hofstadter's G-sequence $h$.

\section{The Misere Version of the Variant of Wythoff's Game}

\begin{definition}\label{miserevariantd1}
Here, we define the misere version of the game in Definition \ref{wythoffvar}. In misere game, the player who plays for the last time loses the game.
In this game 
the player who move into the set $\{(x,y):x+y \leq 2\}$ $=\{(0,0),(1,0),(0,1)$ 
$,(1,1),(2,0),(0,2)\}$ is the loser.
Let $P_2$ be the set of $\mathcal{P}$-positions of this game.
\end{definition}

\begin{figure}[H]
\begin{tabular}{ccc}
\begin{minipage}[t]{0.33\textwidth}
\begin{center}
	\includegraphics[height=2.5cm]{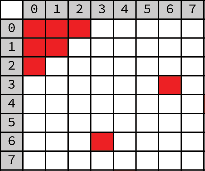}
\caption{Set A}
\label{variantw}
\end{center}
\end{minipage}
\begin{minipage}[t]{0.33\textwidth}
\begin{center}
\includegraphics[height=2.5cm]{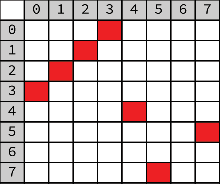}
\caption{Set B}
\label{miserevariant}
\end{center}
\end{minipage}
\begin{minipage}[t]{0.33\textwidth}
\begin{center}
\includegraphics[height=2.5cm]{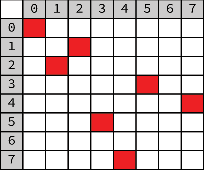}
\caption{Set C}
\label{wythoffpa}
\end{center}
\end{minipage}
\end{tabular}
\end{figure}

\begin{definition}\label{defnofabc}
Let 
$A=\{(0,0),(0,1),(0,2),(1,0),(1,1),(2,0),(3,6),(6,3)\}$,\\
$B=\{(0,3),(1,2),(2,1),(3,0),(4,4),(5,7)$
$,(7,5)\}$, and 
$C=\{(0,0),(1,2),(2,1),$\\
$(3,5),(5,3),(4,7),(7,4)\}$.\\
\end{definition}

\begin{lemma}\label{lemmaforsetsabc}
For Sets $A,B$ and $C$ in Definition \ref{defnofabc}, we obtain the following;\\
$(i)$ $A=P_1 \cap \{(x,y):x,y \leq 7\}$;\\
$(ii)$ $B=P_2 \cap \{(x,y):x,y \leq 7\}$;\\
$(iii)$ $C=P_0 \cap \{(x,y):x,y \leq 7\}$.
\end{lemma}
\begin{proof}
By Definition \ref{newfunction} and Example \ref{exampleforb1b2}, 
$\{(b_1(0),b_2(0)), (b_1(1),b_2(1)), (b_1(2),b_2(2)),$ \\
$(b_1(3),b_2(3))\} \cup \{(0,0),(1,1)\}$
$=\{(0,1), (0,2), (3,6), (4,8)\} \cup \{(0,0),(1,1)\} \subset P_1$, and hence 
we obtain $\mathrm{(i)}$.  We obtain $\mathrm{(iii)}$ directly from Theorems \ref{theoremforwythoffandvar} and  \ref{theoremforwythoff}.

We prove $\mathrm{(ii)}$.
By the definition of the game in Definition \ref{miserevariantd1}, $(3,0),(2,1),(1,2),$ \\ $(0,3) \in P_3$. 
It is clear that these four positions are only $\mathcal{P}$-positions of the game $\{(v,w):v+w \leq 3\}$.
From any position $(x,y)$ such that $4 \leq x+y \leq 7$,
you can move to one of $(3,0),(2,1),(1,2),(0,3)$, and hence 
$(x,y)$ is a $\mathcal{N}$-position when $4 \leq x+y \leq 7$.
From $(4,4)$, you cannot move to any $\mathcal{P}$-position,
and hence $(4,4)$ is a $\mathcal{P}$-position, but from 
any position $(x,y)$ such that $x+y = 8$ and $(x,y) \ne (4,4)$, you can move to a  $\mathcal{P}$-position. Hence, these positions are  $\mathcal{N}$-positions. Similarly, we prove that $(5,7),(7,5)$ are $\mathcal{P}$-positions of the game in Definition \ref{miserevariantd1}.
\end{proof}

\begin{lemma}\label{lemmamove}
Let $x,y \in \mathbb{Z}_{\geq0}$ such that $x \geq 8$ or $y \geq 8$. Then, we obtain the following:\\
$(i)$ if  $y \leq 7$ and $y \ne 6$,  then $M_1(x,y) \cap B \ne \emptyset$, $M_1(x,y) \cap C \ne \emptyset$,  
$M_1(x,6) \cap B = M_1(x,6) \cap C = \emptyset$, and 
$M_2(x,y) \cap C = M_2(x,y) \cap B = \emptyset$;\\
$(ii)$ if $x \leq 7$ and $x \ne 6$, then $M_2(x,y) \cap B \ne \emptyset$, $M_2(x,y) \cap C \ne \emptyset$,  
$M_2(6,y) \cap B = M_2(6,y) \cap C = \emptyset$, and $M_1(x,y) \cap C = M_1(x,y) \cap B = \emptyset$;\\
$(iii)$ if $x \leq y+3$ and $y \leq x+3$, then 
$M_3(x,y) \cap B  \ne \emptyset$, and 
$M_3(x,y) \cap C \ne \emptyset$;\\
$(iv)$ if $x \geq y+4$ or $y \geq x+4$, then $M_3(x,y) \cap B = M_3(x,y) \cap C = \emptyset$.
\end{lemma}
\begin{proof}
This lemma follows directly from Figures \ref{miserevariant} and \ref{wythoffpa}.
\end{proof}

\begin{lemma}\label{lemmamove2}
For $x,y \in \mathbb{Z}_{\geq0}$ such that $x \geq 8$ or $y \geq 8$, we obtain the following:\\ 
$(i)$ $M_1(x,y) \cap B \ne \emptyset$ if and only if $M_1(x,y) \cap C \ne \emptyset$;\\
$(ii)$ $M_2(x,y) \cap B \ne \emptyset$ if and only if $M_2(x,y) \cap C \ne \emptyset$;\\
$(iii)$ $M_3(x,y) \cap B \ne \emptyset$ if and only if $M_3(x,y) \cap C \ne \emptyset$.
\end{lemma}
\begin{proof}
This lemma follows directly from  Lemma \ref{lemmamove}. 
\end{proof}

Theorem \ref{twogamesth} below shows the similarity between the misere version of Wythoff's game and Wythoff's game.
\begin{theorem}\label{twogamesth}
When $x \geq 8$ or $y \geq 8$, a position $(x,y)$ is a 
 $\mathcal{P}$-position of Wythoff's game if and only if it is a  $\mathcal{P}$-position of the game in Definition \ref{miserevariantd1}.
\end{theorem}
\begin{proof}
Let $V_7 = \{(x,y):x,y \leq 7\}$ and  $U_k=\{(x,y):x+y \leq k\}$, and by mathematical induction we prove that 
\begin{equation}
(U_{n}-V_7) \cap P_2 = 
(U_{n}-V_7) \cap P_0  \nonumber
\end{equation}
for any natural number $n$.
Since $(U_{15}-V_7) \subset \{(u,v):u \geq 8 \text{ and } v \leq 7\}$
$\cup \{(u,v):u \leq 7 \text{ and } v \geq 8\}$, by  $(i)$ and $(ii)$ of Lemma \ref{lemmamove}, any point $(x,y) \in U_{15}-V_7$ such that $x \ne 6$ or $y \ne 6$ is a $\mathcal{N}$-position of the game in Definition \ref{miserevariantd1} and Wythoff's game. 
By  $(iii)$ of Lemma \ref{lemmamove}, the set $\{(6,8),(6,9),(8,6),(9,6)\}$  is a set of $\mathcal{N}$-positions of the game in Definition \ref{variantsum} and Wythoff's game. 

Therefore, we obtain 
\begin{equation}
(U_{15}-V_7) \cap P_0 = \emptyset \nonumber
\end{equation}
and
\begin{equation}
(U_{15}-V_7) \cap P_2 = \emptyset. \nonumber
\end{equation}
Therefore, 
$(U_{15}-V_7) \cap P_0 = (U_{15}-V_7) \cap P_2 $.

For some natural number $k$ with $k \geq 16$, we suppose that 
\begin{equation}
(U_{k}-V_7) \cap P_0 = 
(U_{k}-V_7) \cap P_2.   \label{ukv8requa} 
\end{equation}
Let  $x,y \in \mathbb{Z}_{\ge 0}$ such that $(x,y) \in U_{k+1}-V_7$.
Then, for $i=1,2,3$, by Definition \ref{movewythoff}
\begin{equation}
M_i(x,y) \subset U_k,\nonumber
\end{equation}
and hence we have 
\begin{align}
M_i(x,y)\cap P_2& = M_i(x,y) \cap ((U_k-V_7) \cup V_7) \cap P_2 \nonumber \\
& = (M_i(x,y) \cap (U_k-V_7)\cap P_2) \cup (M_i(x,y) \cap V_7\cap P_2)  \nonumber \\
& = (M_i(x,y) \cap (U_k-V_7)\cap P_2) \cup (M_i(x,y) \cap B) \label{m1andb} 
\end{align}
and
\begin{align}
M_i(x,y)\cap P_0& = M_i(x,y) \cap ((U_k-V_7) \cup V_7) \cap P_0 \nonumber \\
& = (M_i(x,y) \cap (U_k-V_7)\cap P_0) \cup (M_i(x,y) \cap V_7\cap P_0)  \nonumber \\
& = (M_i(x,y) \cap (U_k-V_7)\cap P_0) \cup (M_i(x,y) \cap C). \label{m1andb2}
\end{align}      
By Lemma \ref{lemmamove2}, Equations  (\ref{ukv8requa}), (\ref{m1andb}), and (\ref{m1andb2}), we have 
\begin{equation}
M_i(x,y)\cap P_2 = \emptyset \text{ if and only if } M_i(x,y)\cap P_0 = \emptyset \label{mirelation}
\end{equation}
for $i=1,2,3$.
Hence,
\begin{equation}
(U_{k+1}-V_7) \cap P_2 = 
(U_{k+1}-V_7) \cap P_0.   \nonumber
\end{equation}
Therefore, by mathematical induction, we have 
\begin{equation}
(U_{n}-V_7) \cap P_2 = 
(U_{n}-V_7) \cap P_0  \nonumber 
\end{equation}
 for any natural number $n$. 
\end{proof}

\section{The Sum of the Variant of Wythoff's game and a one-pile Nim}

\begin{definition}\label{variantsum}
Applying Definition \ref{sumofgames}, we define the sum of the game in Definition \ref{wythoffvar} and the game of a pile of one stone. 
\end{definition}
We denote the position of the game in Definition \ref{variantsum} by 
 three coordinates $\{x,y,z\}$. The coordinates $x,y$ define the number of stones in the first and second piles, or, if we use a queen in the game, the position of the queen on the chessboard. 
The parameter $z=1$ if there is a stone in the third pile, and $z=0$ if there is no stone in the third pile. Note that when $z=0$, the games in  Definitions \ref{variantsum} and \ref{wythoffvar} are the same game. Let $P_3$ be the set of  $\mathcal{P}$-positions of the game in Definition \ref{variantsum}, and let $M_4(x,y,z)=(x,y,0)$.

\begin{definition}\label{defofp4}
Let $P_4=\{(x,y,1):(x,y) \in P_2\} \cup \{(x,y,0):(x,y) \in P_1\}$.
\end{definition}
Next, we aim to prove that $P_3=P_4$ and use it to prove the last theorem of this paper.
\begin{definition}\label{defnofabc2}
Let $A^{*}=\{(0,0,0),(0,1,0),(0,2,0),(1,0,0),(1,1,0),(2,0,0),(3,6,0),$\\
$(6,3,0)\}$ and 
$B^{*}=\{(0,3,1),(1,2,1),(2,1,1),(3,0,1),(4,4,1),(5,7,1)$
$,(7,5,1)\}$.
\end{definition}

\begin{lemma}\label{lemmaforp4}
For Sets $P_4$, $A^{*}$, and $B^{*}$, we have the following equation.
\begin{equation}
P_4 \cap \{(x,y,z):x,y \leq 7\} = A^{*}\cup B^{*}.\label{p4equalab}
\end{equation}
\end{lemma}
\begin{proof}
By Lemma \ref{lemmaforsetsabc}, and Definitions
\ref{defofp4}, and \ref{defnofabc2}, we obtain Equation $(\ref{p4equalab})$.
\end{proof}

\begin{lemma}\label{thesumlemma}
$P_3 \cap \{(x,y,z):x,y \leq 7\} = A^{*}\cup B^{*}.$
\end{lemma}
\begin{proof}
By Lemma \ref{lemmaforsetsabc}, Definitions \ref{variantsum} and \ref{defnofabc2}, we obtain that 
$(x,y,0) \in P_3 \cap \{(x,y,z):x,y \leq 7\}$ if and only if $(x,y,0) \in A^{*}\cup B^{*}.$

Since $(x,y)$ is a $\mathcal{P}$-position of the game in Definition \ref{wythoffvar} when $x+y \leq 2$, 
we obtain that 
$(3,0),(2,1),(1,2),(0,3)$ are $\mathcal{N}$-position of the game in Definition \ref{wythoffvar}.
Therefore, $\{(3,0,0),(2,1,0),(1,2,0),(0,3,0)\} \cup \{(x,y,1):x,y \leq 2\}$ are a set of $\mathcal{N}$-positions of the game in Definition \ref{variantsum}. Hence, 
 $(3,0,1),(2,1,1),(1,2,1),(0,3,1)$ are $\mathcal{P}$-positions of the game in Definition \ref{variantsum}. From $(4,4,1)$, 
 we cannot move to any of $(3,0,1),(2,1,1),(1,2,1),(0,3,1)$ by $M_i$ for $i=1,2,3$. $M_4(4,4,1)=(4,4,0)$ and $(4,4) \notin A$.
Hence, $(4,4.1)$ is a $\mathcal{P}$-position of the game in Definition \ref{variantsum}. Similarly, we prove that $(5,7,1),(7,5,1)$ are $\mathcal{P}$-positions of the game in Definition \ref{variantsum}.
\end{proof}

\begin{lemma}\label{p4ispposition}
For Set $P_4$ in Definition \ref{defofp4}, we have the following:\\
$(i)$ if $(x,y,z) \in P_4$, $M_i(x,y,z) \cap P_4 = \emptyset$ for $i=1,2,3,4$;\\
$(ii)$  if $(x,y,z) \notin P_4$, $M_i(x,y,z) \cap P_4 \ne \emptyset$ for some $i$.
\end{lemma}
\begin{proof}
By Lemmas \ref{lemmaforp4} and \ref{thesumlemma}, 
\begin{equation}
P_3 \cap \{(x,y,z):x,y \leq 7\}=P_4 \cap \{(x,y,z):x,y \leq 7\}.\nonumber
\end{equation}
Since $P_3$ is the set of $\mathcal{P}$-positions of the game in Definition \ref{variantsum},
we obtain $\mathrm{(i)}$ and $\mathrm{(ii)}$ for $(x,y,z) \in  P_4 \cap \{(x,y,z):x,y \leq 7\}$.

We assume that $x \geq 8$ or $y \geq 8$. Suppose that $(x,y,0) \in P_4$. Then, $(x,y) \in P_1$.
$P_1$ is the set of $\mathcal{P}$-positions of the game in Definition \ref{wythoffvar}, and hence 
 $M_i(x,y) \cap P_1 = \emptyset$ for any $i=1,2,3$. Therefore, we obtain $M_i(x,y,0) \cap P_4 = \emptyset$ for any $i=1,2,3$.

Suppose that $(x,y,0) \notin P_4$. Then, $(x,y) \notin P_1$, and hence  $M_i(x,y,) \cap P_1 \ne \emptyset$ for some $i$.
Therefore, we obtain $M_i(x,y,0) \cap P_4 \ne \emptyset$ for some $i$.

Suppose that  $(x,y,1) \in P_4$. Then, $(x,y) \in P_2$. Since $P_2$ is the set of $\mathcal{P}$-positions of the game in Definition \ref{miserevariantd1}, we have  $M_i(x,y) \cap P_2 = \emptyset$ for $i=1,2,3$. Therefore, we obtain $M_i(x,y,1) \cap P_4 = \emptyset$ for $i=1,2,3$. Since $x \geq 8$ or $y \geq 8$, by Theorem \ref{twogamesth}, $(x,y) \in P_0$.
Since $P_0 \cap P_1 = \emptyset$, $(x,y) \notin P_1$. Therefore, $M_4(x,y,1)=(x,y,0) \notin P_4$.

We assume that $(x,y,1) \notin P_4$. Then, $(x,y) \notin P_2$.
Since $P_2$ is the set of $\mathcal{P}$-positions of the game in Definition \ref{miserevariantd1}, 
$M_i(x,y) \in P_2$ for some $i$ with $1 \leq i \leq 3$. Then, 
$M_i(x,y,1) \in P_4$ for some $i$ with $1 \leq i \leq 3$.
\end{proof}

\begin{definition}\label{newgameforproof}
By Lemma \ref{p4ispposition}, we define a game that has $P_4$ as the set of $\mathcal{P}$-positions and $\{(x,y,0):x+y \leq 2\}$ as a the set of terminal positions.
\end{definition}

\begin{lemma}\label{sumlemma}
The set of $\mathcal{P}$-positions of the game in Definition \ref{newgameforproof} is the same as the set of $\mathcal{P}$-positions of the game in Definition \ref{variantsum}.
\end{lemma}
\begin{proof}
Suppose that there is a position $(x,y,1)$ such that $(x,y,1)$ is a $\mathcal{P}$-position of the game in  
Definition \ref{newgameforproof} and a $\mathcal{N}$-position of the game in  
Definition \ref{variantsum}. Both game have the same $\textit{move}$, and hence
we can move to a position $(u,v,w)$ that is a $\mathcal{N}$-position of the game in  
Definition \ref{newgameforproof} and a $\mathcal{P}$-position of the game in  
Definition \ref{variantsum}. Then, $w=1$, because $(u,w,0)$ is a $\mathcal{P}$-position of the game in  
Definition \ref{newgameforproof} and the game in Definition \ref{variantsum} or
$(u,w,0)$ is a $\mathcal{N}$-position of the game in  
Definition \ref{newgameforproof} and the game in Definition \ref{variantsum}.
By continuing this process, we will enter the area $\{(x,y,z):x,y \leq 7 \}$, but this contradicts 
Lemmas \ref{lemmaforp4} and \ref{thesumlemma}.
\end{proof}

\begin{theorem}\label{twogamesth2}
For a position $(x,y)$ with $x \geq 8$ or $y \geq 8$, the Grundy number of the position $(x,y)$ is $1$ in the game in Definition \ref{wythoffvar} if and only if $(x,y)$ is a $\mathcal{P}$-position of Wythoff's game.
\end{theorem}
\begin{proof}
Suppose that $x \geq 8$ or $y \geq 8$.
By Theorem \ref{twogamesth}, 
\begin{equation}
(x,y) \in P_0 \text{ if and only if } (x,y) \in P_2,\label{condition111}
\end{equation}
and by Definition \ref{defofp4},
\begin{equation}
(x,y) \in P_2 \text{ if and only if } (x,y,1) \in P_4.\label{condition222}
\end{equation}
By Lemma \ref{sumlemma},
\begin{equation}
(x,y,1) \in P_4 \text{ if and only if } (x,y,1) \in P_3.\label{condition333}
\end{equation}
By Definition \ref{variantsum},
$(x,y,1) \in P_3$ if and only if 
the Grundy number of $(x,y)$ of the game in Definition \ref{wythoffvar} is $1$, and hence
by Relations $(\ref{condition111})$, $(\ref{condition222})$, and $(\ref{condition333})$, we finish the proof.
\end{proof}

\end{document}